\documentclass{amsart}

\usepackage{amsmath, amsthm, amssymb}
\usepackage{color}
\usepackage{hyperref}
\newtheorem{theorem}{Theorem}
\newtheorem{lemma}[theorem]{Lemma}

\newcommand{\Rl}{\mathbb{R}}

\newcommand{\Hi}{\operatorname{\mathbf{H}}}

\begin{document}

\title{Long time Solutions for a Burgers-Hilbert Equation via a Modified Energy Method}

\author{John K. Hunter}
\address{ Department of Mathematics, University of California at Davis}
\thanks{The first author was partially supported by the NSF under grant number DMS-0072343.}
\email{hunter@math.davis.edu}
\author{Mihaela Ifrim}
 \address{Department of Mathematics, McMaster University}.
\email{mifrim@math.mcmaster.ca}

\author{ Daniel Tataru}
\address{Department of Mathematics, University of California at Berkeley}
 \thanks{The third author was partially supported by the NSF grant DMS-0801261
as well as by the Simons Foundation}
\email{tataru@math.berkeley.edu}

\author{Tak Kwong Wong}
\address{Department of Mathematics\\ University of Pennsylvania}
\email{takwong@math.upenn.edu}

\date{\today}

\begin{abstract}
We consider an initial value problem for a quadratically nonlinear inviscid
Burgers-Hilbert equation that models the motion of vorticity discontinuities. We use a
modified energy method to prove the existence of small, smooth solutions
over cubically nonlinear time-scales. 
\end{abstract}

\maketitle

%\begin{keywords}
%Nonlinear waves, inviscid Burgers equation, vorticity discontinuities.
%\end{keywords}

%\begin{AMS}
%37L65, 76B47.
%\end{AMS}
%\pagestyle{myheadings}
%\thispagestyle{plain}
\markboth{HUNTER, IFRIM, TATARU, WONG}{BURGERS-HILBERT EQUATION}
%\markright{HUNTER AND TATARU}{BURGERS-HILBERT EQUATION}

%\maketitle
\section{Introduction}

We consider the following initial value problem for an inviscid Burgers-Hilbert equation for $u(t,x)$:
\begin{align}
\label{bheq}
\begin{split}
&u_t+ u u_x = \mathbf{H}\left[ u\right] ,\\
&u(0,x) = u_0(x),
\end{split}
\end{align}
where $\mathbf{H}$ is the spatial Hilbert transform on $\Rl$, and
the initial data $u_0$ is sufficiently small
\begin{equation*}
\Vert u_{0}(x)\Vert _{H^2(\mathbb{R})}\leq \epsilon \ll 1.
\end{equation*}
Here, $H^k(\Rl)$ denotes the standard Sobolev space of functions with $k$ weak $L^2$-derivatives.

Equation (\ref{bheq}) was proposed in \cite{biello} as a model equation for nonlinear waves with constant frequency.
It also provides a formal asymptotic approximation
for the small-amplitude motion of a planar vorticity
discontinuity located at $y = u(t,x)$ \cite{biello, marsden}.
Moreover, even though the equation is quadratically nonlinear, this approximation
is valid over cubically nonlinear time-scales.

Equation (\ref{bheq}) is nondispersive, and solutions of the linearized equation oscillate in time but do not
exhibit any dispersive decay. Nevertheless,
in comparison with the inviscid Burgers equation, small smooth solutions of (\ref{bheq}) have an enhanced, cubically nonlinear lifespan.
This enhanced lifespan was observed numerically in \cite{biello} and proved in \cite{hunter} by use of a change of the independent variable, which was suggested by a normal form transformation of the equation.

In this paper we give a different, and simpler, proof of the enhanced lifespan of smooth
solutions of (\ref{bheq}) from the one in \cite{hunter}.  Our main result is the following theorem:

\begin{theorem}
\label{th:main}
$a)$ Let $k \ge 2 $. Suppose that the initial data $u_0 \in H^k(\Rl)$ for the equation \eqref{bheq} satisfies $$ \Vert u_{0}\Vert _{H^2(\mathbb{R})}\leq \epsilon \ll 1 .$$
Then  there exists a solution $
u \in C\left( I^\epsilon ;{H}^{k}\left( \mathbb{R}\right) \right) \cap C^1 \left(I^\epsilon;{H}^{k-1}\left(\mathbb{R}\right) \right)$
of (\ref{bheq})
defined on the time-interval $I^\epsilon=\left[-{\alpha}/{\epsilon^2},{\alpha}/{\epsilon^2}\right]$, where $\alpha >0$ is a universal constant, so that
\begin{equation*}
\Vert u\Vert _{L^{\infty}(I^{\epsilon}; H^{s}(\mathbb{R}))}\lesssim \Vert u(0)\Vert _{ H^s(\mathbb{R})}, \qquad 0\leq s\leq k.
\end{equation*}
\\
$b)$ Consider $u_1, u_2$ two solutions of \eqref{bheq} as in part $(a)$. Then
\begin{equation*}
\Vert u_1-u_{2}\Vert _{L^{\infty}(I^{\epsilon}; L^2(\mathbb{R}))}\lesssim \Vert u_1(0)-u_{2}(0)\Vert _{ L^2(\mathbb{R})}.
\end{equation*}
\end{theorem}
The quadratically nonlinear terms in (\ref{bheq}) are non-resonant, and a standard approach for problems of this type  is to use a normal form transformation (see for instance \cite{shatah}) to remove them. However, this is a quasi-linear problem, so the normal form transformation is unbounded, and this method cannot be applied directly. Our alternative approach is to use the normal form transformation in order to derive a modified energy functional for the original problem \eqref{bheq} which evolves according to a cubic law.

The same modified energy method also applies to the linearization of the Burgers-Hilbert equation, and allows us to control differences of solutions on a cubic time-scale.

We remark that with some additional work one can replace the $H^2$ norm for the data in part $(a)$ of the above theorem with $H^{s_0}$ for any $s_0> 3/2$. Similarly, the $L^2$ norm in part $(b)$ can be replaced by $H^s$ for a smaller range $0 \leq s \leq s_0 -1$.

\section{The Normal Form Transformation}

Straightforward energy estimates applied to (\ref{bheq}) yield
\[
\frac{d}{dt} \| \partial_x^k u\|_{L^2}^2 \lesssim \|u_x\|_{L^\infty}
\|u\|_{H^k}^2
\]
for $k\ge 2$, which gives a lifespan of smooth solutions of the order $\epsilon^{-1}$,
as in the standard local existence theory for quasi-linear hyperbolic PDEs \cite{kato, majda}.

The quadratically nonlinear terms in (\ref{bheq}) can be  removed by a normal form transformation \cite{biello,hunter}
\begin{equation}
\label{nft}
v = u + \mathbf{H}[\mathbf{H}u\cdot \mathbf{H}u_x ].
\end{equation}
The transformed equation has the form
\begin{equation}
v_{t}+ \mathcal{Q}(u)=\mathbf{H}\left[v \right]
\label{veq}
\end{equation}
where $\mathcal{Q}(u)$ is cubic in $u$ but involves two spatial derivatives.
Straightforward energy estimates for (\ref{veq}) yield
\begin{equation*}
\frac{d}{dt} \| \partial_x^k v\|_{L^2}^2 \lesssim \|u_x\|^2_{L^\infty}
\|u\|_{H^{k+1}}^2.
\end{equation*}
However, since the above normal form transformation is not invertible,
the energy estimates for $v$ do not close.
Thus, the $v$-equation is cubically nonlinear, but there is a loss of derivatives. On the other hand,
the $u$-equation is quadratically nonlinear, but there is no loss of derivatives.

The reason for this disparity in energy estimates is that the $H^k$-norms of $u$ and $v$ are not comparable.
From (\ref{nft}), we have
\begin{equation}
\|\partial_x^k v\|_{L^2}^2 = \|\partial_x^k u\|_{L^2}^2 + 2\langle \partial_x^k u, \partial_x^k \mathbf{H}[\mathbf{H}u \cdot \mathbf{H}u_x]\rangle
+ \|\partial_x^k\mathbf{H}[\mathbf{H}u \cdot \mathbf{H}u_x]\|_{L^2}^2.
\label{vkenergy}
\end{equation}
The second term on the right-hand side of (\ref{vkenergy}) is comparable to the $H^k$-norm of $u$ because
\[
\langle \partial_x^k u, \partial_x^k \mathbf{H}[\mathbf{H}u \cdot \mathbf{H}u_x]\rangle
= \left(k+\frac{1}{2}\right) \langle \mathbf{H} u_x, (\partial_x^k \mathbf{H} u)^2\rangle
+ l.o.t.
\]
where $l.o.t.$ stands for lower order terms. The third term on the right-hand side of (\ref{vkenergy}) is not comparable to the $H^k$-norm, but it is quartic and therefore irrelevant to the cubically nonlinear energy estimates.

This discussion suggests that we define a modified energy functional by dropping the higher-derivative quartic term
from the right-hand side of (\ref{vkenergy}) to get
\begin{equation}
\label{henergy}
E_k(u) = \frac{1}{2} \|\partial_x^k u\|_{L^2}^2 +  \langle \partial_x^k u,\partial_x^k
 \mathbf{H}[\mathbf{H}u \cdot \mathbf{H}u_x]\rangle.
\end{equation}
As we will show, this modified energy is equivalent to the standard $H^k$-energy and satisfies
cubically nonlinear estimates without a loss of derivatives.

\section{The Modified Energy Method}

In this section, we use the modified energy functional introduced in \eqref{henergy} to prove Theorem~\ref{th:main}. We first show that $E_k$ energy is equivalent to the standard $H^k$ energy.

\begin{lemma}
\label{lemma:lemaenergy}
Let $E_k$ be defined by (\ref{henergy}). Then
\begin{equation*}
E_k(u) = \frac{1}{2} \|\partial_x^k u\|_{L^2}^2 (1+ O(\| \mathbf{H}u_x\|_{L^\infty})).
\end{equation*}
\end{lemma}

\begin{proof}
Denote
\begin{equation*}
T_u f := \mathbf{H}\left[ \mathbf{H}u \cdot \mathbf{H}f_x\right] .
\end{equation*}
This operator is essentially skew-adjoint. Moreover, we have
\begin{equation*}
\begin{aligned}
\langle f, T_u f \rangle &= - \langle \mathbf{H}f, \mathbf{H}u \cdot \mathbf{H}f_x \rangle \\
&= \frac{1}{2} \int_{\mathbb{R}} \mathbf{H}u_x \vert \mathbf{H}f\vert ^2 dx = O(\| \mathbf{H}u_x\|_{L^\infty}) \|f\|_{L^2}^2.
\end{aligned}
\end{equation*}

We write
\begin{equation*}
\partial_x^k  \mathbf{H}\left[\mathbf{ H}u \cdot \mathbf{H} u_x\right] = T_u \partial_x^k u + l.o.t.
\end{equation*}
where $l.o.t.$ stands for terms of the form $\mathbf{H} [ \partial_x^j\mathbf{ H}u_x \cdot \partial_x^{k-1-j}\mathbf{H} u_x ]$
with $0 \leq j \leq k-1$. For the leading term we use the previous estimate for $T_u$, and for the lower order terms we use
interpolation,
\[
\| \partial_x^j\mathbf{ H}u_x \cdot \partial_x^{k-1-j}\mathbf{H} u_x\|_{L^2} \lesssim \|\mathbf{H}u_x\|_{L^\infty} \|\partial_x^{k-1} u_x\|_{L^2}
\]
Thus, we obtain
\begin{equation*}
\langle \partial_x^k u,\partial_x^k
 \mathbf{H}\left[ \mathbf{H}u \cdot \mathbf{H} u_x\right] \rangle =  O(\| \mathbf{H}u_x\|_{L^\infty}) \|\partial_x^k u\|_{L^2}
\end{equation*}
\end{proof}
%%%%%%%%%%%%%%%%%%%%%%%%%%%%

The energy estimate is discussed in the following lemma:
\begin{lemma}
For the energy (\ref{henergy}) the estimate below holds: for any $\delta>0$,
\begin{equation*}
\frac{d}{dt} E_k(u) \leq  C_{k,\delta} \Vert u_x\Vert _{H^{\frac12+\delta}}^2\|u\|_{H^k}^2
\end{equation*}
where $C_{k,\delta}$ is a constant depending on $k$ and $\delta$ only.
\end{lemma}

\begin{proof}
A direct computation yields
\begin{equation*}
\begin{aligned}
-\frac{d}{dt} E_k(u) = &\int_{\mathbb{R}}  \partial_x^k (u u_x)  \partial_x^k
 \mathbf{H}\left[ \mathbf{H}u \cdot \mathbf{H}u_x\right] \, dx  \\
 +&\int_{\mathbb{R}} \partial_x^k u \left\lbrace  \partial_x^k
 \mathbf{H}\left[ \mathbf{H}[u u_x]   \cdot \mathbf{H} u_x\right]  +  \partial_x^k
 \mathbf{H}\left[ \mathbf{H}u  \cdot \mathbf{H} [u u_x]_x\right] \right\rbrace   dx .
 \end{aligned}
\end{equation*}
Note that all cubic terms cancel. This is because the energy was
determined using the normal form transformation, which at the level of
the energy amounts to eliminating the cubic terms.
 Using the antisymmetry of $H$ and integrating by parts
  we rewrite the above expression in the form
\[
-\frac{d}{dt} E_k(u) = \int_{\mathbb{R}}  - \partial_x^k \Hi (u u_x)  \partial_x^k
 \left[ \mathbf{H}u \cdot \mathbf{H}u_x\right]  + \partial_x^k \Hi u_x  \partial_x^k
 \left[ \mathbf{H}u \cdot \Hi (u u_x) \right] dx.
\]
We observe that in each term in the expansion of the above integrand
has the form
\[
\Hi \partial^{\alpha} u  \Hi \partial^{\beta} u \Hi (\partial^\gamma u\  \partial^\delta u)
\]
where
\[
\alpha + \beta + \gamma + \delta =  2k+2
\]
We observe that as long as
\[
1 \leq \alpha, \beta, \gamma, \delta \leq k, \qquad \alpha+\beta \geq 3,
\delta+ \gamma \geq 3
\]
such terms can be estimated directly in terms of the right hand side
of the energy relation in the Lemma. We will call such terms ``good''.

We return to the expression above, and observe that a cancellation occurs
if the second  $\partial_x^k$ operator in each of the two expressions applies
to the second factor. Using a full binomial expansion, we are left with the
expressions
\[
I_\alpha =  \int_{\mathbb{R}}  - \partial_x^k \Hi (u u_x)  \cdot \partial_x^{\alpha}
  \mathbf{H}u \cdot  \partial_x^{k-\alpha}\mathbf{H}u_x  +
\partial_x^k \Hi u_x \cdot
  \partial_x^{\alpha}\mathbf{H}u \cdot  \partial_x^{k-\alpha} \Hi (u u_x)  dx
\]
where $1 \leq \alpha \leq k$.  Distributing the remaining derivatives,
we get good terms if any derivatives fall on the $u$ factor
(an additional integration by parts  is needed to see this for the second term).
Thus we have
\[
I_\alpha =  \int_{\mathbb{R}}  - \Hi ( u  \partial_x^k u_x)  \cdot \partial_x^{\alpha}
  \mathbf{H}u \cdot  \partial_x^{k-\alpha}\mathbf{H}u_x  +
\partial_x^k \Hi u_x \cdot
  \partial_x^{\alpha}\mathbf{H}u \cdot  \Hi (u   \partial_x^{k-\alpha} u_x)  dx
+ good.
\]
If $u$ is commuted out the two expressions above cancel. Hence
\[
I_\alpha =  \int_{\mathbb{R}}  - [\Hi,u]  \partial_x^k u_x  \cdot \partial_x^{\alpha}
  \mathbf{H}u \cdot  \partial_x^{k-\alpha}\mathbf{H}u_x  +
\partial_x^k \Hi u_x \cdot
  \partial_x^{\alpha}\mathbf{H}u \cdot  [\Hi ,u ]  \partial_x^{k-\alpha} u_x  dx
+ good.
\]

The commutator $[\mathbf{H},u] f$ vanishes unless the frequency of $u$
is at least as large as the frequency of $f$.  Hence we can always
move derivatives from $f$ onto $u$.
Then for the first term in $I_{\alpha}$ we can use directly the commutator estimate below:
\begin{equation}\label{com}
\Vert \left[\mathbf{H},u \right]\partial_{x}^{k}u_x\Vert_{L^2}\lesssim\Vert  u_x \Vert _{L^{\infty}}\Vert \partial_{x}^ku  \Vert _{L^{2}}.
\end{equation}
This is a consequence of the commutator estimates result obtained by Dawson,
McGahagan, and Ponce in \cite{ponce}. We recall it below:
\begin{lemma} Let $\mathbf{H}$ denote the Hilbert transform. Then for
  any $p\in (1, \infty)$ and any $l, m\in \mathbb{Z}^+$ there exists
  $c=c(p,l,m) > 0$ such that
\begin{equation*}
\Vert \partial ^{l}_{x}[\mathbf{H}, a]\partial_{x}^{m}f\Vert_{L^{p}}\leq c\Vert \partial_{x}^{l+m}a\Vert _{L^{\infty}}\Vert f \Vert_{L^{p}}.
\end{equation*}
\end{lemma}

For the second in $I_{\alpha}$ we need to integrate once by parts since the first factor
has too many derivatives. If $\alpha < k$ this is done directly. If $\alpha = k$
then we move a derivative on the third factor. We are left with two terms which are nontrivial, namely
\[
J_k := \int |\partial_x^k \Hi u|^2  [\Hi ,u_x] u_x dx
\]
and
\[
J_{k-1} := \int |\partial_x^k \Hi u|^2  [\Hi ,u] u_{xx} dx.
\]
For these we need the pointwise bound
\begin{equation}\label{danny}
\|  [\Hi ,u_x] u_x\|_{L^\infty} + \|[\Hi ,u] u_{xx}\|_{L^\infty} \lesssim \|u_x\|_{H^{\frac12+\delta}}^2.
\end{equation}
Indeed, both terms on the left can be estimated in $H^{\frac12+\delta}$ using standard Littlewood-Paley
decompositions.
\end{proof}

By a Gronwall type argument, the two lemmas above lead to an $\epsilon^{-2}$ lifespan, proving part (a) of Theorem~\ref{th:main}.

%%%%%%%%%%%%%%%%%%%%%%%%%%%%%%%%%%%%%%%%%%%%%%%%%%%%%%%%%%%%%%%%%%%%%%%%%%%%%%%%%%%%%%%%%%%%%%%%%%%%%%%%%%%%%%%%
Next, in order to prove part (b) of Theorem~\ref{th:main}, we consider the linearized Burgers-Hilbert equation:
\begin{equation}
\label{linearization}
w_t+wu_x+uw_x=\mathbf{H}w.
\end{equation}

Beginning with the normal form transformation  (\ref{nft}) for the Burgers-Hilbert equation, we obtain the normal form transformation  for the linearization (\ref{linearization}):
\begin{equation*}
q:=w+\vert \partial_x\vert (\mathbf{H}w\cdot\mathbf{H}u).
\end{equation*}
where $\vert \partial_x\vert = \mathbf{H}\,\circ \, \partial_{x}$.
Inspired by this we can define a modified linearized energy
\begin{equation}
\label{lin energy}
E_{lin}(w)=
\int w^2+2\vert \partial_x\vert w\cdot \mathbf{H}u\cdot\mathbf{H} w\,dx
\end{equation}
This is equivalent to the $L^2$ norm of $w$,
\[
E_{lin}(w) = (1+O(\|\Hi u_x\|_{L^\infty}) \|w\|_{L^2}^2.
\]
This is proved using integration by parts in order to write
\begin{equation*}
E_{lin}(w)
=\int w^2-\vert \partial_{x}\vert u\cdot \left( \mathbf{H}w\right) ^2\,dx.
\end{equation*}

We want to prove that $E_{lin}$ satisfies good cubic bounds; a Gronwall type  argument implies part $(b)$ of Theorem~\ref{th:main}.
\begin{lemma}
For the energy (\ref{lin energy}) the estimate below holds: for any $\delta>0$,
\begin{equation*}
\frac{d}{dt} E_{lin}(w) \leq C_{k,\delta}  \Vert u_x\Vert _{H^{\frac12+\delta}}^2 \|w\|_{L^2}^2
\end{equation*}
where $C_{k,\delta}$ is a constant depending on $k$ and $\delta$ only.
\end{lemma}

\begin{proof}
A straightforward computation leads to
\begin{equation*}
\begin{aligned}
\frac{d}{dt}E_{lin}(w)&=\int 2\mathbf{H}w \cdot\mathbf{H}\left[wu_x \right]\cdot \mathbf{H}u_x+ 2\mathbf{H}w \cdot\mathbf{H}\left[uw_x \right]\cdot \mathbf{H}u_x
\\
&\hskip1in+ \vert \mathbf{H}w\vert^2 \cdot\vert \partial_{x}\vert (uu_x) \, dx\\
&=\int 2\mathbf{H}w \cdot\mathbf{H}\left[wu_x \right]\cdot \mathbf{H}u_x+2\mathbf{H}w\cdot \left[\mathbf{H},u \right]\partial_{x}w\cdot \mathbf{H}u_x
\\
&\hskip1in +
\vert \mathbf{H}w\vert^2 \partial_{x} \left( \left[ \mathbf{H}, u \right] \partial_{x}u \right)  \, dx.
\end{aligned}
\end{equation*}
The first term is estimated directly and for the second we use the commutator bound \eqref{com}. The third term can be controlled because we can rewrite
\[
\partial_{x} \left( \left[ \mathbf{H}, u \right] \partial_{x}u \right) = \left[ \mathbf{H}, u_x \right] \partial_{x}u + \left[ \mathbf{H}, u \right] \partial_{x}^2u,
\]
where the right hand side can be estimated by the pointwise bound \eqref{danny}.
\end{proof}


\begin{thebibliography}{10}
\bibitem{biello} J.~Biello and J.~K.~Hunter, \emph{Nonlinear Hamiltonian waves with constant frequency and
surface waves on vorticity discontinuities},  Comm. Pure Appl. Math. \textbf{63}, 2009, 303-336.

\bibitem{hunter} J.~K.~Hunter and M.~Ifrim, \emph{Enhanced lifespan of smooth solutions of a Burgers-Hilbert equation},  SIAM J. Math. Anal., \textbf{44}, 2012,
 2039-2052.

\bibitem{kato} T.~Kato, \emph{The Cauchy problem for quasi-linear symmetric hyperbolic systems}, Arch. Rational Mech. Anal. \textbf{58}, 1975, 181-205.

\bibitem{majda} A. J.~Majda, \emph{Compressible fluid flow and systems of conservation laws in several space variable}, Springer-Verlag,
 Volume \textbf{53}, 1984.


\bibitem{marsden} J.~Marsden and A.~Weinstein, \emph{Coadjoint orbits, vortices, and Clebsch variables
for incompressible fluids}, Physica D\textbf{7}, 1983, 305-323.

\bibitem{ponce} G.~ Ponce, F.~ Linares and D.~ Pilod, \emph{Well-Posedness for a higher-order Benjamin-Ono equation.} Proc. AMS. \textbf{136}, 2008,
 2081-2090.

\bibitem{shatah} J.~Shatah, \emph{Normal forms and quadratic nonlinear Klein-Gordon nonlinearities}, Comm. Pure Appl. Math. \textbf{38}, 1985, 685-696.

\end{thebibliography}
\end{document}